\documentclass[12pt,oneside]{amsart}

\usepackage{amssymb,latexsym}
\usepackage{mathtools}
\usepackage{enumerate}  
\usepackage{a4wide}    

\usepackage[colorlinks=true,citecolor=blue,linkcolor=blue]{hyperref}
\usepackage{latexsym,amsmath,amssymb}
\usepackage{accents}
\usepackage{color}  
\title{Three examples of Sharp Commutator Estimates via Harmonic Extensions}

\author{Armin Schikorra}
\address[Armin Schikorra]{Department of Mathematics,
University of Pittsburgh,
301 Thackeray Hall,
Pittsburgh, PA 15260, USA}
\email{armin@pitt.edu}

plus 6pt minus 12pt
\newtheorem{theorem}{Theorem}
\newtheorem{lemma}[theorem]{Lemma}

\theoremstyle{definition}
\newtheorem{remark}[theorem]{Remark}

\def\divv{{\rm div\,}}

\def\curl{{\rm curl\,}}
\def\lip{{\rm Lip\,}}

\renewcommand{\div}{\divv}
\newcommand{\R}{\mathbb{R}}

\newcommand{\brac}[1]{\left (#1 \right )}

\newcommand{\Ep}{\bigwedge\nolimits}

\newcommand{\barint}{
\rule[.036in]{.12in}{.009in}\kern-.16in \displaystyle\int }

\newcommand{\barcal}{\mbox{$ \rule[.036in]{.11in}{.007in}\kern-.128in\int $}}

\def\mvint_#1{\mathchoice
          {\mathop{\vrule width 6pt height 3 pt depth -2.5pt
                  \kern -8pt \intop}\nolimits_{\kern -3pt #1}}%
          {\mathop{\vrule width 5pt height 3 pt depth -2.6pt
                  \kern -6pt \intop}\nolimits_{#1}}%
          {\mathop{\vrule width 5pt height 3 pt depth -2.6pt
                  \kern -6pt \intop}\nolimits_{#1}}%
          {\mathop{\vrule width 5pt height 3 pt depth -2.6pt
                  \kern -6pt \intop}\nolimits_{#1}}}

\newcommand{\lap}{\Delta }
\newcommand{\laph}{\laps{1}}
\newcommand{\aleq}{\lesssim}

\newcommand{\aeq}{\approx}
\newcommand{\Rz}{\mathcal{R}}

\newcommand{\laps}[1]{(-\lap) ^{\frac{#1}{2}}}

\newcommand{\lapms}[1]{I^{#1}}

\numberwithin{figure}{section}

\begin{document}
\begin{abstract}
Recently, Lenzmann and the author observed how to obtain a large class of sharp commutator estimates by a combination of an integration by parts, an harmonic extension, and trace space estimates. In this survey we review this approach in three concrete examples: the Jacobian estimate by Coifman-Lions-Meyer-Semmes, the Coifman-Rochberg-Weiss commutator estimate for Riesz transforms, and a Kato-Ponce-Vega-type inequality.
\end{abstract}

\maketitle 
\tableofcontents

\section{An estimate by Coifman-Lions-Meyer-Semmes}\label{ch:clms}
Throughout this text, we will only consider maps which are smooth and have compact support, i.e. $C_c^\infty$-maps. We will make no attempt to obtain an optimal space in the estimates which we consider hold. Rather, our focus lies on obtaining optimal estimates, which by density arguments may lead to these optimal spaces.

Let $u \in C_c^\infty(\R^n,\R^n)$ and $\varphi \in C_c^\infty(\R^n)$. The Jacobian of $u$, sometimes denoted by ${\rm Jac}(u)$ is the determinant of the gradient
\[
{\rm Jac}(u) = \det(\nabla u).
\]
The Jacobian naturally appears in geometric contexts, since it describes the volume of a square distorted by the linear map $\nabla u(x)$ -- as we know from the transformation rule for integrals.

The following estimates are then quite obvious\footnote{Here and henceforth by $A \aleq B$ we mean that $A \leq C B$ for some constant $C$ which is always supposed to not depend on $A$ or $B$ (or other \emph{relevant} quantities).}
\begin{equation}\label{eq:jacstupid}
\int_{\R^n} \det(\nabla u) \varphi \aleq \|\nabla u\|_{L^n}^n\, \|\varphi\|_{L^\infty}
\end{equation}
and (by an integration by parts)
\[
\int_{\R^n} \det(\nabla u) \varphi \aleq \|u\|_{L^\infty}\, \|\nabla u\|_{L^n}^{n-1}\, \|\nabla \varphi\|_{L^n}
\]
But these are not sharp estimates. 

This had somewhat been known for quite some time in the theory of geometric PDEs (again: the Jacobian is a very geometric object and appears for example in surfaces of prescribed mean curvature) \cite{Wente-1969,Reshetnyak-1967}, but it took until the 1990s to really understand the reason (in the sense of Harmonic Analysis). After an earlier result by M\"uller \cite{Mueller-1990} (who proved $L\log L$-estimates for the Jacobian), Coifman-Lions-Meyer-Semmes \cite{CLMS93} obtained the following remarkable estimate
\begin{equation}\label{eq:clms}
\int_{\R^n} \det(\nabla u) \varphi \aleq \|\nabla u\|_{L^n}^{n}\, [\varphi]_{BMO}
\end{equation}
Here $[\varphi]_{BMO}$ denotes the seminorm of the space of function of bounded mean oscillation (BMO), namely
\[
 [\varphi]_{BMO} := \sup_{r > 0, x_0 \in \R^n} \mvint_{B_r(x_0)} \left |\varphi - \mvint_{B_r(x_0)} \varphi \right |.
\]
Estimate \eqref{eq:clms} is a strictly weaker estimate than \eqref{eq:jacstupid}, since 
\[
[\varphi]_{BMO} \aleq \|\varphi\|_{L^\infty}
\]
Two different methods are given in \cite{CLMS93} in order to obtain \eqref{eq:clms}:
\begin{enumerate}
\item Showing that $\det(\nabla u)$ belongs to the Hardy space $\mathcal{H}^1(\R^n)$ if $\nabla u \in L^n(\R^n,\R^n)$ and then using duality with $BMO$
\item Reduction to the Coifman-Rochberg-Weiss commutator, see section~\ref{ch:CRW}
\end{enumerate}
Before proceeding, let us stress that \eqref{eq:clms} has been a crucial tool for regularity theory of geometric PDEs, such as regularity theory for the equations for surfaces of prescribed mean curvature and for harmonic maps into manifolds, see for example \cite{Helein90,Helein91,Bethuel-1992,Riviere-2007}, see also the monograph  \cite{Helein-book} and for some open problems the survey \cite{SS17}. The usefulness of \eqref{eq:clms} is based on the continuous embedding of the Sobolev space $W^{1,n}$ into $BMO$ (and there is no embedding of $W^{1,n}$ into $L^\infty$ for $n \neq 1$!), which makes the following estimate a consequence of \eqref{eq:clms}
\begin{equation}\label{eq:clmsbmoemb}
\int_{\R^n} \det(\nabla u) \varphi \aleq \|\nabla u\|_{L^n}^{n}\, \|\nabla \varphi\|_{L^n}
\end{equation}
Actually, in terms of Lorentz space one can improve this estimate: since $W^{1,(n,\infty)}$ embeds into $BMO$,
\begin{equation}\label{eq:clmsbmoembinfty}
\int_{\R^n} \det(\nabla u) \varphi \aleq \|\nabla u\|_{L^n}^{n}\, \|\nabla \varphi\|_{L^{(n,\infty)}}
\end{equation}
\subsection{``Intermediate'' sharp estimates for the Jacobians}
The estimate \eqref{eq:clms}, \eqref{eq:clmsbmoemb} can be interpreted also as a distributional definition of the Jacobian (cf. \cite{BN11}): $\det(\nabla u)$ is well defined as an element of $BMO^\ast$ or $(W^{1,n})^\ast$. 

But this can be ``improved'' in the differential order of the Sobolev spaces: If in \eqref{eq:clms} one ``allows more derivatives'' to ``fall'' on $\varphi$, one can \emph{uniformly} reduce the derivatives that ``fall'' on $u$. Namely, the following estimate is true
\begin{equation}\label{eq:clmsintermediate}
\int \det(\nabla u) \varphi \aleq [u^1]_{\dot{W}^{s_1,p_1}}\, \ldots [u^n]_{\dot{W}^{s_n,p_n}}\, [\varphi]_{\dot{W}^{s_{n+1},p_{n+1}}}
\end{equation}
holds whenever $s_1,\ldots,s_{n+1} > 0$ and $p_1,\ldots,p_{n+1} \in (1,\infty)$ are so that 
\begin{equation}\label{eq:sin}
\sum_{i=1}^{n+1} s_i = n
\end{equation}
and
\begin{equation}\label{eq:pin}
\sum_{i=1}^{n+1} \frac{1}{p_i} = 1
\end{equation}
Estimates of this sort were observed not so long after the work of \cite{CLMS93}, see e.g. \cite{SY99}; Indeed, one can hope to obtain this from multilinear interpolation of the inequality \eqref{eq:clms}. It seems however that some versions of estimates of the form \eqref{eq:clmsintermediate} were known to some experts even \emph{earlier} than the work in \cite{CLMS93}  -- e.g. the technique in \cite{Tartar-1982} hints into this direction. This seems to be the case due to the fact that \eqref{eq:clmsintermediate} is technically \emph{easier} than \eqref{eq:clms}. We shall make the last (very superficial) statement more precise below: we consider proofs by Tartar \cite{Tartar-1982}, by Brezis-Nguyen \cite{Tartar-1982}, and Lenzmann and the author \cite{LS16} of different versions of \eqref{eq:clmsintermediate}. We then show that, in order to prove the limit space BMO-inequality \eqref{eq:clms}, one more push is needed.

\subsection{``Intermediate'' estimates: An argument due to Tartar}
The technique for proving the so-called Wente's inequality in \cite{Tartar-1982} inspire the following argument.

Denote by $\mathcal{F}$ the Fourier transform. For simplicity of presentation we restrict ourselves to the case $n=2$, but this arguments takes over to any dimension. The properties of the Fourier transform (products in geometric space become convolutions in phase space, derivatives in geometric space become polynomials in phase space) imply that one can write
\[
\mathcal{F} \det(\nabla u^1 | \nabla u^2)(\xi) = c\, \int \det (\xi, \xi-\eta) \mathcal{F}u^1(\xi-\eta)\, \mathcal{F}u^2(\eta)\, d\eta
\]
where $c$ is a (real) number. The compensation effect that is responsible for the correctness of estimates such as \eqref{eq:clms}, \eqref{eq:clmsintermediate} etc. is the following: By the properties of the determinant we have 
\[
\det(\xi,\xi-\eta) = -\det(\xi,\eta) = \det(\eta,\xi-\eta).
\]
In particular, the following estimates are true:
\[ 
|\det(\xi,\xi-\eta)| \aleq \begin{cases}
|\xi|\, |\xi-\eta|\\
|\xi|\, |\eta|\\
|\eta|\, |\xi-\eta|
\end{cases}
\]
Interpolating between these three options, for any $s_1,s_2,s_3 > 0$, $s_1 + s_2 + s_3 = 2$ (cf. \eqref{eq:sin}) we have 
\begin{equation}\label{eq:detest}
|\det(\xi,\xi-\eta)| \aleq |\xi-\eta|^{s_1}\, |\eta|^{s_2}\, |\xi|^{s_3} 
\end{equation}
So we set 
\[
\kappa(\xi,\eta) := |\xi-\eta|^{-s_1}\, |\eta|^{-s_2}\, |\xi|^{-s_3}\, \det(\xi,\xi-\eta),
\]
which smooth away from $\eta = 0$ and $\xi = 0$ and $\eta = \xi$ and satisfies 
\[
|\kappa(\xi,\eta)| \aleq 1.
\]
Define the bilinear operator $T = T(a,b)$ as 
\[
\mathcal{F}(T(a,b)) := \int \kappa(\xi, \eta) \mathcal{F} a(\xi-\eta)\, \mathcal{F} b(\eta)\, d\eta
\]
and use the Plancherel Theorem to find
\[
\int \det(\nabla u^1 | \nabla u^2) \varphi = \int T(\laps{s_1}u^1, \laps{s_2} u^2)\, \laps{s_3} \varphi
\]
In some sense $T$ is a zero-multiplier operator, see \cite[Theorem 5.1.]{Tomita2010}, so one obtains the estimate 
\[
\begin{split}
&\int \det(\nabla u^1 | \nabla u^2)\, \varphi \\
=& \int T(\laps{s_1}u^1, \laps{s_2} u^2)\, \laps{s_3} \varphi\\
\leq & \|T(\laps{s_1}u^1, \laps{s_2} u^2)\|_{L^{p_3'}}\, \|\laps{s_3} \varphi\|_{L^{p_3}}\\
\aleq & \|\laps{s_1} u^1\|_{L^{p_1}}\, \|\laps{s_2} u^2\|_{L^{p_2}}\, \|\laps{s_3} \varphi\|_{L^{p_3}} 
\end{split}
\]
This is not exactly the same estimate as in \eqref{eq:clmsintermediate}, since $W^{s,p}$ is not characterized by the norm $\|\laps{s} f\|_{L^p}$ unless $p =2$ -- but it clearly goes into the right direction -- and with a bit more care (and para-products) one can obtain \eqref{eq:clmsbmoembinfty} from this strategy. \qed

\begin{remark}
\begin{itemize} 
\item Observe that it does not seem to be obvious how this possibly could lead to the BMO-estimate \eqref{eq:clms} for $s_3 \to 0$ .
 \item In order to avoid multilinear Fourier multipliers one can conduct the argument described above also without Fourier transform. Instead, one can use the representation
 \[
  \nabla u(x) = \Rz \laph u(x) = c\int_{\R^n} \frac{(x-y)}{|x-y|^{n+1}} \laph u(y)\, dy.
 \]
Now a similar estimate to \eqref{eq:detest} can be used for the kernels $\frac{(x-y)}{|x-y|^{n+1}}$ instead of the Fourier symbol $i\xi/|\xi|$. This was used, for ``intermediate estimates'' of some commutators in \cite{Schikorra11}, see also \cite{DLS14,BRS16}.
\end{itemize}
\end{remark}
\subsection{``Intermediate'' estimates for  \texorpdfstring{$s_1 = \ldots = s_{n+1}=\frac{n}{n+1}$}{s1=...=s(n+1)}: A proof due to Brezis-Nguyen}\label{s:BN}
The following is a beautiful idea by Brezis and Nguyen \cite{BN11} for \[s:= s_1 = \ldots, s_{n+_1} = \frac{n}{n+1}\] and \[p:= p_1 = \ldots = p_n = n+1.\]

Denote by \[\R^{n+1}_+ = \R^n \times (0,\infty)\] 
and from now on we adapt the notation that $x \in \R^n$ and $t \in (0,\infty)$, i.e. variables in $\R^{n+1}$ are $(x,t)$.

Let $U: \R^{n+1}_+ \to \R^n$ be an extension of $u: \R^n \to \R^n$, and $\Phi: \R^{n+1}_+ \to \R^n$ be an extension of $\varphi: \R^n \to \R$. 

Then, by Stokes' theorem (identifying $\R^n$ with $\R^n \times \{0\} = \partial \R^{n+1}_+$),
\[
\left |\int_{\R^n} \det(\nabla_x u^1,\ldots,\nabla_x u^n)\, \varphi\, dx\right |
= \left |\int_{\R^{n+1}_+} \det(\nabla_{x,t} U^1,\ldots,\nabla_{x,t} U^n, \nabla_{x,t} \Phi)\, d(x,t) \right |.
\]
Here $\nabla_x = (\partial_{x^1},\ldots,\partial_{x^n})$ denotes the gradient for functions in $\R^n$, and $\nabla_{x,t} = (\partial_{x^1},\ldots,\partial_{x^n},\partial_t)$ denotes the gradient for functions in $\R^{n+1}_+$. From the above equality we obtain by by H\"older's inequality,
\begin{equation}\label{eq:hoelder}
\left |\int_{\R^n} \det(\nabla_x u^1,\ldots,\nabla_x u^n)\, \varphi\, dx\right | \leq [U^1]_{\dot{W}^{1,n+1}(\R^{n+1}_+)}\, \ldots\, [U^n]_{\dot{W}^{1,n+1}(\R^{n+1}_+)}\, [\Phi]_{\dot{W}^{1,n+1}(\R^{n+1}_+)}.
\end{equation}
This estimate holds for \emph{any} extension $U^1,\ldots,U^n, \Phi: \R^{n+1}_+ \to \R$ of $u^1,\ldots,u^n,\varphi: \R^{n} \to \R$. In particular it holds for extensions $U_i$ that (approximately) realize the trace embedding \[W^{1,n+1}(\R^{n+1}_+) \hookrightarrow W^{\frac{n}{n+1},n+1} (\R^n),\]
namely for extensions $U$ of $u$ such that
\begin{equation}\label{eq:Urealizesu}
[U]_{W^{1,n+1}(\R^{n+1}_+)} \aeq [u]_{ W^{\frac{n}{n+1},n+1}(\R^n)}.
\end{equation}
For example, the \emph{harmonic extension} 
\begin{equation} 
\label{eq:harmext} U(x,t) = p_t \ast u(x)
\end{equation} gives \eqref{eq:Urealizesu}, where $p_t$ is the Poisson kernel
\begin{equation} 
\label{eq:poissonkernel}
p_t(z) = c \frac{t}{\brac{|z|^2 + t^2}^{\frac{n+1}{2}}}.
\end{equation}
But, (under certain assumptions on the integrability and decay, i.e. on $s$) also kernels of the form
\begin{equation} 
\label{eq:sharmext} 
p^s_t(z) = c \frac{t^s}{\brac{|z|^2 + t^2}^{\frac{n+s}{2}}}.
\end{equation}
satisfy \eqref{eq:Urealizesu}. 

Whatever choice for the extension we make, once \eqref{eq:Urealizesu} is satisfied we have obtained 
\[
\left |\int_{\R^n} \det(\nabla_x u^1,\ldots,\nabla_x u^n)\, \varphi\, dx\right | \leq [u^1]_{\dot{W}^{\frac{n}{n+1},n+1}(\R^{n+1}_+)}\, \ldots\, [u^n]_{\dot{W}^{\frac{n}{n+1},n+1}(\R^{n+1}_+)}\, [\Phi]_{\dot{W}^{\frac{n}{n+1},n+1}(\R^{n+1}_+)}
\]
which is \eqref{eq:clmsintermediate} for our special choice.\qed
\subsection{``Intermediate estimates'': general case }
Here we follow \cite{LS16} to obtain \eqref{eq:clmsintermediate} in full generality by the harmonic extension. 

By adapting in the above argument \eqref{eq:hoelder} the H\"older inequality it is easy to obtain 
\[\left |\int_{\R^n} \det(\nabla_x u^1,\ldots,\nabla_x u^n)\, \varphi\, dx\right | \leq [u^1]_{\dot{W}^{s_1,p_1}(\R^{n+1}_+)}\, \ldots\, [u^n]_{\dot{W}^{s_n,p_n}(\R^{n+1}_+)}\, [\Phi]_{\dot{W}^{s_{n+1},p_{n+1}}(\R^{n+1}_+)}
\]
for $p_i$ satisfying \eqref{eq:pin} and 
\[
s_i := 1-\frac{1}{p_i},
\]
that is, for trace spaces of $W^{1,p_i}$. But what to do for estimates in spaces $W^{s,p}$ which are not trace spaces of $W^{1,q}$, i.e. for $s \neq 1-\frac{1}{p_i}$? Weights in $t$-direction are the answer. We can smuggle those in by writing with the help of \eqref{eq:sin}, \eqref{eq:pin}
\[
1 = t^{1-s_1 - \frac{1}{p_1}}\cdot \ldots \cdot t^{1-s_{n+1} - \frac{1}{p_{n+1}}}.
\]
Then, from the argument in Section~\ref{s:BN}, we obtain
\[
\begin{split}
&\left |\int_{\R^n} \det(\nabla_x u^1,\ldots,\nabla_x u^n)\, \varphi\, dx\right |\\ 
\leq& \|t^{1-s_1 - \frac{1}{p_1}} \nabla_{x,t} U^1\|_{L^{p_1}(\R^{n+1}_+)}\, \ldots\, \|t^{1-s_n - \frac{1}{p_n}}\nabla_{x,t} U^n\|_{L^{p_n}(\R^{n+1}_+)}\, \|t^{1-s_{n+1} - \frac{1}{p_{n+1}}}\nabla_{x,t} \Phi\|_{L^{p_{n+1}}(\R^{n+1}_+)}
\end{split}
\]
Again this inequality holds for all possible extensions $U^1,\ldots,U^n, \Phi: \R^{n+1}_+ \to \R$ of $u^1,\ldots,u^n,\varphi: \R^{n} \to \R$, and we need to find an extension such that 
\begin{equation}\label{eq:weightrace}
\|t^{1-s - \frac{1}{p}} \nabla_{x,t} U\|_{L^{p_1}(\R^{n+1}_+)} \aeq [u]_{ W^{s,p}(\R^n)}.
\end{equation}
We are lucky: under some integrability and decay assumptions extensions such as the one defined in \eqref{eq:sharmext}, and in particular the harmonic extension \eqref{eq:harmext}, satisfy \eqref{eq:weightrace}. The proof of this fact is somewhat scattered throughout the literature: an early work where this appears is \cite{U69}, see also \cite{MR2015}; it also is partially contained (somewhat hidden) in Stein's books, e.g. \cite{S93}. As a special case the $s$-harmonic extension theory was popularized in the 2010s in the PDE community by Caffarelli and Silvestre \cite{CS07}. In terms of Besov- and Triebel spaces the most general statement known to the author is due to Bui and Candy \cite{BC17}. 
Thus, by the right choice of extension (for example the harmonic extension), we obtain \eqref{eq:clmsintermediate} in its full generality.\qed

Again, one should notice that it is in no way obvious how $s_{n+1} \to 0$ implies the BMO-estimate \eqref{eq:clms}. This is what we meant after Equation \eqref{eq:pin} when we said that the $BMO$-estimate \eqref{eq:clms} is \emph{structurally more complex} than the ``intermediate'' estimate \eqref{eq:clmsintermediate}.
In the next section, we shall see what additional trick we need: it's an additional integration by parts.
\subsection{The BMO-estimate}
In this section we prove the $BMO$-estimate \eqref{eq:clms} first obtained in \cite{CLMS93}. More precisely, we show the estimate
\begin{equation}\label{eq:clmsp}
\int_{\R^n} \det(\nabla u) \varphi \aleq \|\nabla u^1\|_{L^{p_1}} \ldots \|\nabla u^n\|_{L^{p_n}}\, [\varphi]_{BMO}
\end{equation}
holds whenever $p_1,\ldots,p_n \in (1,\infty)$ so that 
\[
\sum_{i=1}^n \frac{1}{p_i} = 1.
\]
The proof of the BMO-estimate as in \cite{LS16} follows from an adaption of the above arguments in three directions: Firstly, we will apply an additional integration by parts in $t$-direction, namely the almost trivial observation that for sufficient decay at $\infty$ a smooth function $f$ satisfies
\begin{equation}\label{eq:pit}
\int_{0}^\infty f(t) dt = -\int_{0}^\infty t\, \partial_t f(t) dt. 
\end{equation}
Secondly, we will (for the first time) use the harmonicity of the extension: if $\lap_{x,t} F \equiv 0$ then obviously
\[
\partial_{tt} F = - \lap_x F.
\]
We will use this fact essentially only in order to replace derivatives in $t$-direction (which we \emph{cannot} integrate by parts in $\R^{n+1}_+$ since there would appear boundary terms) by derivatives in $x$-directions (which we \emph{can} integrate by parts in $\R^{n+1}_+$ without having boundary terms).
Thirdly, we will need a replacement for the trace estimate such as \eqref{eq:Urealizesu}, \eqref{eq:weightrace} for BMO: Carleson measure estimates.

\subsubsection{An additional integration by parts}
As always, let $U: \R^{n+1}_+ \to \R^n$ be an extension of $u: \R^n \to \R^n$, and $\Phi: \R^{n+1}_+ \to \R^n$ be an extension of $\varphi: \R^n \to \R$. 

By Stokes' theorem, as before,
\begin{equation}\label{eq:stokesbmo}
\mathcal{C} := \left |\int_{\R^n} \det(\nabla_x u^1,\ldots,\nabla_x u^n)\, \varphi\, dx\right |
= \left |\int_{\R^{n+1}_+} \det(\nabla_{x,t} U^1,\ldots,\nabla_{x,t} U^n, \nabla_{x,t} \Phi)\, d(x,t) \right |.
\end{equation}
Now we perform an additional integration by parts in $t$-direction, namely \eqref{eq:pit}.
\[\mathcal{C} = \left |\int_{\R^{n+1}_+} t\, \partial_t \det(\nabla_{x,t} U^1,\ldots,\nabla_{x,t} U^n, \nabla_{x,t} \Phi)\, d(x,t) \right |.
\]
We claim that \emph{if $U^i$, $\Phi$ are harmonic}, then it is possible to estimate $\mathcal{C}$ by
\begin{equation}\label{eq:goalclms}
\mathcal{C} \aleq \sum_{i=1}^n \int_{\R^{n+1}_+} t\, |\nabla_{x,t} U^1|\ldots |\nabla_{x,t} \nabla_x U^i| \ldots |\nabla_{x,t} U^n|\, |\nabla_{x,t} \Phi|
\end{equation}
That is, a second derivative hits one of the $U$'s and it does so in $x$-direction.
\begin{proof}[Proof of \eqref{eq:goalclms}] It might be interesting to observe that the following argument does not use the determinant structure anymore. It simply follow by the product structure of the integral. The determinant structure was only important for the first integration by parts \eqref{eq:stokesbmo}.

Assume that $U^i$ and $\Phi^i$ are harmonic. We split the integral in $n+1$ parts,
\[
\mathcal{C} \leq I_1 + \ldots + I_n + II 
\]
where for $i = 1,\ldots, n$,
\[
I_{i} := \left |\int_{\R^{n+1}_+} \det(\nabla_{x,t} U^1,\ldots, \partial_t \nabla_{x,t} U^i,\ldots ,\nabla_{x,t} U^n, \nabla_{x,t} \Phi)\, d(x,t) \right |,
\]
and
\[
II := \left |\int_{\R^{n+1}_+} t\, \det(\nabla_{x,t} U^1,\ldots,\nabla_{x,t} U^n, \nabla_{x,t} \partial_t \Phi)\, d(x,t) \right |,
\]
For $i=1,\ldots,n$ harmonicity implies that 
\[
|\partial_t \nabla_{x,t} U^i| = |(\partial_{x^1} \partial_t U^i,\ldots,\partial_{x^n} \partial_t U^i, - \lap_x U^i)| \leq |\nabla_{x,t}\, \nabla_x U^i|,
\]
so for $I_1,\ldots,I_n$ the estimate \eqref{eq:goalclms} is immediate.

For $II$ we have, again by harmonicity,
\[
\det(\nabla_{x,t} U^1,\ldots,\nabla_{x,t} U^n, \nabla_{x,t} \partial_t \Phi) 
=
\det\left (
\begin{array}{c|c|c|c}
\nabla_{x,t} U^1& \ldots & \nabla_{x,t} U^n& \begin{array}{c} \nabla_{x} \partial_t \Phi\\
-\lap_x \Phi \end{array}
\end{array}
 \right ) 
\]
Thus, $II$ can be estimated by (for a second we write $z = (x,t)$)
\[
\begin{split}
II \leq& \sum_{i_1,\ldots,i_n = 1}^{n+1} \sum_{j=1}^n \left |\int_{\R^{n+1}_+} t\, \partial_{z_{i_1}} U^1\cdot \ldots \cdot \partial_{z_{i_n}} U^n\cdot \partial_{x_{j}} \partial_t \Phi \right |\\
&+ \sum_{i_1,\ldots,i_n = 1}^{n+1} \left |\int_{\R^{n+1}_+} t\, \partial_{z_{i_1}} U^1\cdot \ldots \cdot \partial_{z_{i_n}} U^n\cdot \lap_x \Phi \right |.
\end{split}
\] 
With an integration by parts in $x$-direction (there are no boundary terms in $x$-direction, which is the big difference to integration by parts in $t$-direction)
\[
\begin{split}
II \leq& \sum_{i_1,\ldots,i_n = 1}^{n+1} \sum_{j=1}^n \left |\int_{\R^{n+1}_+} t\, \partial_{x} \brac{\partial_{z_{i_1}} U^1\cdot \ldots \cdot \partial_{z_{i_n}} U^n}\cdot \partial_t \Phi \right |\\
&+ \sum_{i_1,\ldots,i_n = 1}^{n+1} \left |\int_{\R^{n+1}_+} t\, \nabla_x \brac{\partial_{z_{i_1}} U^1\cdot \ldots \cdot \partial_{z_{i_n}} U^n}\cdot \nabla_x \Phi \right |.
\end{split}
\]
Both terms satisfy the estimate that we claimed \eqref{eq:goalclms}.
\end{proof}

The reason we want \eqref{eq:goalclms} is that we find below a square function, and will use the square function estimate \cite[section I, \textsection 8.23, p.46]{S93} which states that (``tangential'' version)
\begin{equation}\label{eq:squaret}
\brac{\int_{\R^n} \brac{\int_{t=0}^\infty |\kappa_t \ast f(x)|^2\, \frac{dt}{t}}^p\, dx}^{\frac{1}{p}} \aleq \|f\|_{L^p(\R^n)}
\end{equation}
and (``non-tangential'' version)
\[
\brac{\int_{\R^n} \brac{\int_{|x-y| < t} |\kappa_t \ast f(y)|^2\, \frac{dt\, dy}{t^{n+1}}}^p \, dx}^{\frac{1}{p}} \aleq \|f\|_{L^p(\R^n)}
\]
hold true for kernels $\kappa$ with sufficient decay at infinity and $\int \kappa = 0$. In our case we will apply this to $f = \nabla_x u$, and $\kappa = \nabla_{x,t}\Big |_{t =1} p_t$, i.e. we use that
\[
\nabla_{x,t} \nabla_x U = (\nabla_{x,t} p_t) \ast \brac{\nabla_x u} = t\, \kappa_t \ast \nabla_x u.
\]
These are the trace estimates we treat in the next section.
\subsubsection{Trace estimates}\label{s:traceestjacob}
We have found in the last section that if $U^i$ and $\Phi: \R^{n+1}_+ \to \R$ denote the harmonic extensions (with decay to zero at infinity making them unique) of $u^i$ and $\varphi: \R^n \to \R$, then
\begin{equation}\label{eq:goalclms2}
\int_{\R^n} \det(\nabla u^1,\ldots,\nabla u^n)\, \varphi \aleq \sum_{i=1}^n \int_{\R^{n+1}_+} t\, |\nabla_{x,t} U^1|\ldots |\nabla_{x,t} \nabla_x U^i| \ldots |\nabla_{x,t} U^n|\, |\nabla_{x,t} \Phi|
\end{equation}
An important tool is now the characterization of $\varphi \in BMO$ in terms of the harmonic extension $\Phi$. The following theorem follows e.g. from \cite[IV, \textsection 4.3, Theorem 3. pp.159]{S93} or \cite[Theorem 7.3.8.]{GrafakosMF}
\begin{theorem}[Characterization of BMO by Carleson measures]
Let $\Phi: \R^{n+1}_+ \to \R$ be the ($s$-)harmonic extension of $\varphi \in C_c^\infty(\R^{n})$ as in \eqref{eq:harmext}, \eqref{eq:sharmext}. Then
\[
[\varphi]_{BMO} \aeq \brac{|B|^{-1} \sup_{B} \int_{T(B)} t|\nabla_{x,t} \Phi|^2\, dx\, dt}^{\frac{1}{2}}.
\]
Here the supremum is taken over balls $B \subset R^n$ and $T(B) \subset \R^{n+1}_+$ denotes the \emph{tent} over $B$, i.e. if $B = B(x_0,r)$ then 
\[
T(B) = \left \{(x,t) \in \R^{n+1}_+: |x-x_0| < t-r \right \}.
\] 
\end{theorem}
Also, we need the following estimate which serves as a replacement for the $L^1$-$L^\infty$ H\"older inequality on $\R^{n+1}_+$. 
\[
 \int_{\R^{n+1}_+} t\, F(x,t)\, G(x,t)\, dx\, dt\aleq
 \]
 \[\sup_{B \subset \R^n \mbox{ balls} } \brac{|B|^{-1} \int_{T(B)} t| F(y,t)|^2 dy\, dt}^{\frac{1}{2}}\ \int_{\R^n} \brac{\int_{|x-y|<t} |G(y,t)|^2 \frac{dy dt}{t^{n-1}} }^{\frac{1}{2}}dx
 \]
For a proof, see \cite[IV,§4.4, Proposition, p. 162]{S93}.
In particular, if $\Phi: \R^{n+1}_+ \to \R$ is the harmonic extension of $\varphi: \R^n \to \R$ we have 
\begin{equation}\label{eq:jacob1}
 \int_{\R^{n+1}_+} t\, |\nabla_{x,t} \Phi|\, |G(x,t)|\, dx\, dt\aleq
[\varphi]_{BMO} \int_{\R^n} \brac{\int_{|x-y|<t} |G(y,t)|^2 \frac{dy dt}{t^{n-1}} }^{\frac{1}{2}}dx
 \end{equation}
In our situation \eqref{eq:goalclms2} we employ this estimate with
\[
G_i = |\nabla_{x,t} U^1| \ldots  |\nabla_{x,t} \nabla_x U^i| \ldots |\nabla_{x,t} U^n| 
\]
Moreover one can show \cite[(10.3)]{LS16} that if $U: \R^{n+1}_+ \to \R$ is the harmonic extension of $u \in C_c^\infty(\R^n)$, then for all $x \in \R^n$,
\begin{equation}\label{eq:maxfctest}
\sup_{(y,t): |x-y| < t} |\nabla_{x,t} U(y,t)| \aleq \mathcal{M} |\nabla u|(x) + \mathcal{M} |\laps{1} u|(x)
\end{equation}
Here $\mathcal{M}$ is the Hardy-Littlewood maximal function.
Thus,
\[
\begin{split}
&\int_{\R^n} \brac{\int_{|x-y|<t} |G_1(y,t)|^2 \frac{dy dt}{t^{n-1}} }^{\frac{1}{2}}dx\\
\leq &\sum_{D_i \in \{\laph, \nabla_x\}}\int_{\R^n} \mathcal{M} |D_2 u^2|\, \ldots \mathcal{M} |D_n u^n|\,  \brac{\int_{|x-y|<t} |\nabla_{x,t} \nabla_x U^1(y,t)|^2 \frac{dy dt}{t^{n-1}} }^{\frac{1}{2}}dx.\\
\end{split}
\]
H\"older's inequality and the boundedness of maximal functions and Riesz transforms on $L^p$ leads to 
\[
\begin{split}
&\int_{\R^n} \brac{\int_{|x-y|<t} |G_1(y,t)|^2 \frac{dy dt}{t^{n+1}} }^{\frac{1}{2}}dx\\
\leq & \|\nabla u^2\|_{L^{p_2}(\R^n)} \ldots \|\nabla u^n\|_{L^{p_n}(\R^n)}
\brac{\int_{\R^n} \brac{\int_{|x-y|<t} |\nabla_{x,t} \nabla_x U^1(y,t)|^2 \frac{dy dt}{t^{n-1}} }^{\frac{p_1}{2}}dx}^{\frac{1}{p_1}}
\end{split}
\]
Now, we use the non-tangential square function estimate, see \cite[section I, §8.23, p.46]{S93}, which states that for $p \in (1,\infty)$,
\begin{equation}\label{eq:ntsquare}
\brac{\int_{\R^n} \brac{\int_{|x-y| < t} |\kappa_t \ast f(y)|^2\, \frac{dt\, dy}{t^{n+1}}}^p \, dx}^{\frac{1}{p}} \aleq \|f\|_{L^p(\R^n)}
\end{equation}
hold true for kernels $\kappa$ with sufficient decay at infinity and $\int \kappa = 0$. Since $U^1$ is harmonic with decay to zero at infinity, it can be written as $U^1 = p_t \ast u^1$, where $p_t$ is the Poisson kernel as in \eqref{eq:harmext}. Consequently,
\[
\nabla_{x,t} \nabla_x U^1 = (\nabla_{x,t} p_t) \ast \nabla_x u^1,
\]
and just by computing $\nabla_{x,t} p_t$ we can find a map $\kappa$ with sufficient decay (and since it is a derivative with $\int \kappa = 0$) so that 
\[
\nabla_{x,t} p_t(z) = t^{-1} t^{-n}\kappa(z/t) \equiv t^{-1} \kappa_t(z)
\]
Thus,
\[
\begin{split}
&\brac{\int_{\R^n} \brac{\int_{|x-y|<t} |\nabla_{x,t} \nabla_x U^1(y,t)|^2 \frac{dy dt}{t^{n-1}} }^{\frac{p_1}{2}}dx}^{\frac{1}{p_1}}\\
=&\brac{\int_{\R^n} \brac{\int_{|x-y|<t} t^{-2} |\kappa_t \ast \nabla_x u^1(y,t)|^2 \frac{dy dt}{t^{n-1}} }^{\frac{p_1}{2}}dx}^{\frac{1}{p_1}}\\
=&\brac{\int_{\R^n} \brac{\int_{|x-y|<t}  |\kappa_t \ast \nabla_x u^1(y,t)|^2 \frac{dy dt}{t^{n+1}} }^{\frac{p_1}{2}}dx}^{\frac{1}{p_1}}\\
\overset{\eqref{eq:ntsquare}}{\aleq}& \|\nabla_x u^1\|_{L^{p_1}(\R^n)}.
\end{split}
\]
Plugging all these estimates together, we obtain \eqref{eq:clmsp}
\subsection{The actual div-curl estimate}
The theorem by Coifman-Lions-Meyer-Semmes \cite{CLMS93} actually treats div-curl estimates, namely for vectorfields $f: \R^n \to \R^n$, $g: \R^n \to \R^n$ so that 
\[
\div f = 0, \quad \curl g =0
\]
we have for any $p \in (1,\infty)$
\[
\int_{\R^n} f \cdot g\, \varphi \aleq \|f\|_{L^{p}(\R^n)}\, \|g\|_{L^{p'}(\R^n)}\, [\varphi]_{BMO}\\
\]
One can easily obtain the same estimate (and the related intermediate estimates) by the same method as above if one represents (by the Poincar\`e-Lemma) the vector fields as differential forms $f = d\alpha \in C_c^\infty(\Ep^1 \R^n)$, $g = \ast d\beta \in C_c^\infty(\Ep^{n-1} \R^n)$
\[
\left |\int_{\R^n} f \cdot g\, \varphi \right |\equiv \left |\int_{\R^n} da \wedge db\, \varphi\right |
\]
Now the Stokes theorem implies for extensions $A, B, \Phi$ of $a,b,\varphi$ respectively that
\[
\left |\int_{\R^n} da \wedge db\, \varphi \right |= \left |\int_{\R^{n+1}_+} dA \wedge dB\wedge  d\Phi \right |
\] 
The further estimates are exactly as in the above sections.

Let us remark, that an argument based on the harmonic extension argument has been used by Chanillo in \cite{Chanillo91} quite some time ago in the realm of compensated compactness. In particular, for div-curl quantities he obtained estimates of the form \eqref{eq:clmsbmoemb} in this way. The BMO-estimate via this argument seems to be new in \cite{LS16}.

\section{Coifman-Rochberg-Weiss Commutator}\label{ch:CRW}
Let $\Rz = (\Rz_1,\ldots,\Rz_n)$ denote the (vectorial) Riesz transform, given by the Fourier symbol
\[
 \mathcal{F}(\Rz f)(\xi) := c\, i\frac{\xi}{|\xi|}\, \mathcal{F}f(\xi).
\]
where $c$ is a real nonzero number. In \cite{CRW76} Coifman-Rochberg-Weiss proved the following\footnote{Indeed, they proved this estimate for general Calderon-Zygmund operators instead of only the Riesz transforms} estimate for any $p \in (1,\infty)$
\begin{equation}\label{eq:CRW}
 \int_{\R^n} \brac{\Rz_i(f)\, g + f\, \Rz_i(g)}\, \varphi \aleq [\varphi]_{BMO}\, \|f\|_{L^p(\R^n)}\, \|g\|_{L^{p'}(\R^n)}.
\end{equation}
This is a commutator estimate, since it can by duality it is equivalent to
\[
 \|[\varphi,\Rz_i](f)\|_{L^p(R^n)} \aleq [\varphi]_{BMO}\, \|f\|_{L^p(\R^n)},
\]
where
\[
 [\varphi,\Rz](f) = \varphi \Rz_i(f) - \Rz_i(\varphi f).
\]
As for the Jacobian, \eqref{eq:CRW} is an improvement of the (almost) trivial estimate 
\begin{equation}\label{eq:CRWstupid}
 \int_{\R^n} \brac{\Rz_i(f)\, g + f\, \Rz_i(g)}\, \varphi \aleq \|\varphi\|_{L^\infty}\, \|f\|_{L^p(\R^n)}\, \|g\|_{L^{p'}(\R^n)}.
\end{equation}
Indeed, it was shown in \cite{CLMS93} that the Jacobian estimate from section~\ref{ch:clms} follows from \eqref{eq:CRW}.
In \cite{LS16} the estimate \eqref{eq:CRW} is proven by the extension method, namely we obtain
\begin{theorem}
Let $f,g,\varphi \in C_c^\infty(\R^n)$, $i=1,\ldots,n$.
The term
\[
 \int_{\R^n} \brac{\Rz_i(f)\, g + f\, \Rz_i(g)}\, \varphi 
\]
can be estimated by
\begin{enumerate}
 \item The Coifman-Rochberg-Weiss \cite{CRW76} estimate, for any $p \in (1,\infty)$ 
 \[[\varphi]_{BMO}\, \|f\|_{L^p(\R^n)}\, \|g\|_{L^{p'}(\R^n)},\]
 \item For any $s \in (0,1)$ and $p_1,p_2,p_3 \in (1,\infty)$
 \[
 \|\laps{s} \varphi \|_{L^{p_1}(\R^n)} \|\lapms{s} f\|_{L^{p_2}(\R^n)}\, \|g\|_{L^{p_3}(\R^n)}.
 \]
\end{enumerate}

\end{theorem}

The main additional observation in addition to the arguments from section~\ref{ch:clms} is the following:

for a map $f: \R^n \to \R$ denote $f^h: \R^{n+1}_+ \to \R$ the harmonic extension $f^h(x,t) := p_t \ast f(x)$ for the Poisson kernel \eqref{eq:poissonkernel}. Then we have the following for some constant $c \in \R$
\begin{equation}\label{eq:ptrz1}
 \partial_t \brac{\Rz_i f}^h = -c\, \partial_{x_i} f^h,
\end{equation}

\subsection{The integration by parts}
We use the following formula which holds e.g. for any $C^1$-function $\eta : [0,\infty) \to \R$ with sufficient decay at infinity, namely $\lim_{t \to \infty} |\eta(t)| = \lim_{t \to \infty} |\eta'(t)| = 0$:
\[
 \eta(0) = \int_0^\infty t\, \partial_{tt} \eta(t)\, dt.
\]
Let $F, G, \Phi: \R^{n+1}_+$ be the harmonic extensions of $f,g,\varphi: \R^n \to \R$. By an abuse of notation we shall write
\[
 \Rz_i F := (\Rz_i f)^h.
\]
Then we find
\[
\begin{split}
  &\left |\int_{\R^n} \brac{\Rz_i(f)\, g + f\, \Rz_i(g)}\, \varphi \right |\\
  =&\left |\int_{\R^{n+1}_+} t \partial_{tt}\brac{\brac{\Rz_i(F)\, G + F\, \Rz_i(G)}\, \Phi} \right | 
\end{split}
  \]
Our goal is to show at least one of these derivatives hits $\Phi$,
\begin{equation}\label{eq:CRWgoal}
\begin{split}
  &\left |\int_{\R^n} \brac{\Rz_i(f)\, g + f\, \Rz_i(g)}\, \varphi \right | \\
  \aleq & \sum \int_{\R^{n+1}_+} t \brac{|\nabla_{x,t} \tilde{F}|\, |\tilde{G}| + |\tilde{F}|\, |\nabla_{x,t} \tilde{G}|} |\nabla_{x,t} \Phi|
\end{split}
  \end{equation}
where the sum is over $\tilde{F} \in \{ \Rz F, F \}$ and $\tilde{G} \in \{ \Rz G,G\}$.

We compute
\begin{equation}\label{eq:crw:split}
 \begin{split}
 &\partial_{tt}\brac{\brac{\Rz_i(F)\, G + F\, \Rz_i(G)}\, \Phi}\\
 =& \partial_t\brac{\Rz_i(F)\, G + F\, \Rz_i(G)}\, \partial_{t}\Phi\\
 &+\brac{\Rz_i(F)\, G + F\, \Rz_i(G)}\, \partial_{tt}\Phi\\
 &+ \partial_{tt}\brac{\Rz_i(F)\, G + F\, \Rz_i(G)}\, \Phi 
 \end{split}
\end{equation}
Clearly the first term is already of the form we need to get \eqref{eq:CRWgoal}. As for the second term, we can use the harmonicity of $\Phi$,
\[
 \brac{\Rz_i(F)\, G + F\, \Rz_i(G)}\, \partial_{tt}\Phi = -\brac{\Rz_i(F)\, G + F\, \Rz_i(G)}\, \lap_x \Phi
\]
and thus by an integration by parts in $x$-direction
\[
 \int_{\R^{n+1}_+} t \brac{\Rz_i(F)\, G + F\, \Rz_i(G)}\, \partial_{tt}\Phi  = \int_{\R^{n+1}_+} t \nabla_x \brac{\Rz_i(F)\, G + F\, \Rz_i(G)}\, \nabla_x\Phi 
\]
which is of the form \eqref{eq:CRWgoal}.

It remains to compute the last term in \eqref{eq:crw:split}. We have
\[
\begin{split}
  &\partial_{tt}\brac{\Rz_i(F)\, G + F\, \Rz_i(G)}\\
  =& \brac{\partial_{tt}\Rz_i(F)\, G + F\, \partial_{tt}\Rz_i(G)}\\
  &+\brac{\Rz_i(F)\, \partial_{tt}G + \partial_{tt}F\, \Rz_i(G)}\\
  &+2\brac{\partial_t \Rz_i(F)\, \partial_{t}G + \partial_{t}F\, \partial_t \Rz_i(G)}.
\end{split}
  \]
Now we employ \eqref{eq:ptrz1} (we pretend $c =1$ for simplicity).
\[
\begin{split}
  =& \brac{\partial_{tt}\Rz_i(F)\, G + F\, \partial_{tt}\Rz_i(G)}\\
  &+\brac{\Rz_i(F)\, \partial_{tt}G + \partial_{tt}F\, \Rz_i(G)}\\
  &-2\brac{\partial_{x^i} F\, \partial_{t}G + \partial_{t}F\, \partial_{x^i} G}\\
\end{split}
\]
and use the product rule on the last term (factoring $\partial_{x_i}$)
\[
  \begin{split}
  =& \brac{\partial_{tt}\Rz_i(F)\, G + F\, \partial_{tt}\Rz_i(G)}\\
  &+\brac{\Rz_i(F)\, \partial_{tt}G + \partial_{tt}F\, \Rz_i(G)}\\
  &-2\partial_{x^i}\brac{F\, \partial_{t}G + \partial_{t}F\, G}\\
  &+2\brac{F\, \partial_{t}\partial_{x^i}G + \partial_{t}\partial_{x^i}F\, G}\\
\end{split}
  \]
and again by \eqref{eq:ptrz1} we find
\[
  \begin{split}
  =& \brac{\partial_{tt}\Rz_i(F)\, G + F\, \partial_{tt}\Rz_i(G)}\\
  &+\brac{\Rz_i(F)\, \partial_{tt}G + \partial_{tt}F\, \Rz_i(G)}\\
  &-2\partial_{x^i}\brac{F\, \partial_{t}G + \partial_{t}F\, G}\\
  &-2\brac{F\, \partial_{tt}\Rz_i G + \partial_{tt}\Rz_i F\, G}\\
\end{split}
  \]
and thus
\[
  \begin{split}
  =& -\brac{\partial_{tt}\Rz_i(F)\, G + F\, \partial_{tt}\Rz_i(G)}\\
  &+\brac{\Rz_i(F)\, \partial_{tt}G + \partial_{tt}F\, \Rz_i(G)}\\
  &-2\partial_{x^i}\brac{F\, \partial_{t}G + \partial_{t}F\, G}.
\end{split}
  \]
Now we use the harmonicity of $F$ and $G$ (and recall that $\Rz_i F$ and $\Rz_i G$ are by definition also harmonic),
\[
  \begin{split}
  =& +\brac{\lap_x \Rz_i(F)\, G + F\, \lap_x \Rz_i(G)}\\
  &-\brac{\Rz_i(F)\, \lap_x G + \lap_x F\, \Rz_i(G)}\\
  &-2\partial_{x^i}\brac{F\, \partial_{t}G + \partial_{t}F\, G}.
\end{split}
  \]
Next we factor the divergence
\[
  \begin{split}
  =& +\nabla_x \cdot \brac{\nabla_x \Rz_i(F)\, G + F\, \nabla_x \Rz_i(G)}\\
  & -\brac{\nabla_x \Rz_i(F)\cdot \nabla_x G + \nabla_x F\cdot \nabla_x \Rz_i(G)}\\
  &-\nabla_x \cdot \brac{\Rz_i(F)\, \nabla_x G + \nabla_x F\, \Rz_i(G)}\\
  &+\brac{\nabla_x\Rz_i(F)\cdot \nabla_x G + \nabla_x F\cdot \nabla_x \Rz_i(G)}\\
  &-2\partial_{x^i}\brac{F\, \partial_{t}G + \partial_{t}F\, G}\\
\end{split}
  \]
We see that the second and fourth row cancel, and thus
\[
  \begin{split}
  =& +\nabla_x \cdot \brac{\nabla_x \Rz_i(F)\, G + F\, \nabla_x \Rz_i(G)}\\
  &-\nabla_x \cdot \brac{\Rz_i(F)\, \nabla_x G + \nabla_x F\, \Rz_i(G)}\\
  &-2\partial_{x^i}\brac{F\, \partial_{t}G + \partial_{t}F\, G}\\
\end{split}
  \]
But this implies that also for the third term in \eqref{eq:crw:split} we obtain the estimate \eqref{eq:CRWgoal} by an integration by parts.
\subsection{The trace theorems}
We have found in \eqref{eq:CRWgoal}
\[
\begin{split}
  &\left |\int_{\R^n} \brac{\Rz_i(f)\, g + f\, \Rz_i(g)}\, \varphi \right | \\
  \aleq & \sum \int_{\R^{n+1}_+} t \brac{|\nabla_{x,t} \tilde{F}|\, |\tilde{G}| + |\tilde{F}|\, |\nabla_{x,t} \tilde{G}|} |\nabla_{x,t} \Phi|
\end{split}
  \]
where the sum is over $\tilde{F} \in \{ \Rz F, F \}$ and $\tilde{G} \in \{ \Rz G,G\}$.
Now we need to prove trace estimates
\begin{lemma}
Let $F, G, \Phi: \R^{n+1}_+ \to \R$ be the harmonic extensions of $f,g,\varphi : \R^n \to \R$, respectively. Then
\[
 \int_{\R^{n+1}_+} t |\nabla_{x,t} F|\, |G| |\nabla_{x,t} \Phi|
\]
can be estimated by
\begin{enumerate}
\item for $p \in (1,\infty)$ \begin{equation}\label{eq:CRW:tracebmo}
\aleq [\varphi]_{BMO} \|f\|_{L^p}\, \|g\|_{L^{p'}}.
\end{equation}
\item for $p_1,p_2, p_3$ in $(1,\infty)$ with $\frac{1}{p_1} +\frac{1}{p_2}+\frac{1}{p_3}=1$
\begin{equation}\label{eq:CRW:tracei1}
\aleq \|\laps{s} \varphi\|_{L^{p_1}} \|f\|_{L^{p_2}}\, \|\lapms{s} g\|_{L^{p_3}}.
\end{equation}
\begin{equation}\label{eq:CRW:tracei2}
\aleq \|\laps{s} \varphi\|_{L^{p_1}} \|\lapms{s} f\|_{L^{p_2}}\, \|g\|_{L^{p_3}}.
\end{equation}
\end{enumerate}
\end{lemma}
\begin{proof}
To prove \eqref{eq:CRW:tracebmo} we proceed the same way as in the $BMO$-estimate for the Jacobian, Section~\ref{s:traceestjacob}.

For \eqref{eq:CRW:tracei1} use a different version of the maximal function estimate \eqref{eq:maxfctest}, namely we have
\[
 \sup_{t > 0} t |\nabla_{x,t} F(x,t)| \aleq \mathcal{M} f(x).
\]
Thus, by H\"older's inequality
\[
\begin{split}
 &\int_{\R^{n+1}_+} t |\nabla_{x,t} F|\, |G| |\nabla_{x,t} \Phi|\\
 \aleq & \int_{\R^{n}_+} \mathcal{M} f(x)\, \brac{\int_0^\infty t^{2s-1} |G|^2\, dt}^{\frac{1}{2}} \brac{\int_0^\infty t^{1-2s} |\nabla_{x,t} \Phi|^2\, dt}^{\frac{1}{2}}\, dx\\
 \aleq & \|f\|_{L^{p_2}}\, \brac{\int_{\R^n} \brac{\int_0^\infty t^{2s-1} |G|^2\, dt}^{\frac{p_3}{2}}\, dx}^{\frac{1}{p_3}} 
 \brac{\int_{\R^n} \brac{\int_0^\infty t^{1-2s}|\nabla_{x,t} \Phi|^2\, dt}^{\frac{p_1}{2}}\, dx}^{\frac{1}{p_1}}\\
\end{split}
 \]
Now we can write
\[
 \nabla_{x,t} \Phi = t^{s-1}\, \kappa \ast \laps{s} \varphi
\]
where $s < 1$ ensures that $\kappa$ satisfies $\int \kappa = 0$. Thus, we find again a square function estimate, as in \eqref{eq:squaret}, and have 
\[
 \brac{\int_{\R^n} \brac{\int_0^\infty t^{1-2s}|\nabla_{x,t} \Phi|^2\, dt}^{\frac{p_1}{2}}\, dx}^{\frac{1}{p_1}} = \brac{\int_{\R^n} \brac{\int_0^\infty |\kappa_t \ast \laps{s} \varphi|^2\, \frac{dt}{t}}^{\frac{p_1}{2}}\, dx}^{\frac{1}{p_1}} \aleq \|\laps{s} \varphi \|_{L^{p_1}}.
\]
As for $G$, we can write \[G = p_t \ast \laps{s} \lapms{s} g =: t^{-s} \kappa_t \ast \lapms{s} g\]
and use the same square function estimate to obtain 
\[
 \brac{\int_{\R^n} \brac{\int_0^\infty t^{2s-1} |G|^2\, dt}^{\frac{p_3}{2}}\, dx}^{\frac{1}{p_3}}
 = \brac{\int_{\R^n} \brac{\int_0^\infty |\kappa_t \ast \lapms{s} g|^2\, \frac{dt}{t}}^{\frac{p_3}{2}}\, dx}^{\frac{1}{p_3}} \aleq \|\lapms{s} g \|_{L^{p_3}}.
\]
This establishes \eqref{eq:CRW:tracei1}.

For \eqref{eq:CRW:tracei2} we argue similarly,
\[
\begin{split}
 &\int_{\R^{n+1}_+} t |\nabla_{x,t} F|\, |G| |\nabla_{x,t} \Phi|\\
 \aleq & \int_{\R^{n}_+} \mathcal{M} g(x)\, \brac{\int_0^\infty t^{2s+1} |\nabla_{x,t} F|^2\, dt}^{\frac{1}{2}} \brac{\int_0^\infty t^{1-2s} |\nabla_{x,t} \Phi|^2\, dt}^{\frac{1}{2}}\, dx\\
 \aleq & \|g\|_{L^{p_3}}\, \brac{\int_{\R^n} \brac{\int_0^\infty t^{2s+1} |\nabla_{x,t} F|^2\, dt}^{\frac{p_2}{2}}\, dx}^{\frac{1}{p_2}} 
 \brac{\int_{\R^n} \brac{\int_0^\infty t^{1-2s}|\nabla_{x,t} \Phi|^2\, dt}^{\frac{p_1}{2}}\, dx}^{\frac{1}{p_1}}\\
\end{split}
 \]
The term involving $\Phi$ is estimated as above, for $F$ we write 
\[
 \nabla_{x,t} F =: t^{-1-s} \kappa_t \ast \lapms{s} f,
\]
and have by the square function estimate 
\[
 \brac{\int_{\R^n} \brac{\int_0^\infty t^{2s+1} |\nabla_{x,t} F|^2\, dt}^{\frac{p_2}{2}}\, dx}^{\frac{1}{p_2}} 
 = \brac{\int_{\R^n} \brac{\int_0^\infty |\kappa_t \ast \lapms{s} f|^2\, \frac{dt}{t}}^{\frac{p_2}{2}}\, dx}^{\frac{1}{p_2}}  \aleq \|\lapms{s} f\|_{L^{p_2}}.
\]
This establishes \eqref{eq:CRW:tracei2}.
\end{proof}

\section{Coifman-McIntosh-Meyer and Kato-Ponce-Vega type estimates}
In the above section we estimated commutators in $L^p$-spaces. A class of commutator estimates usually called Coifman-McIntosh-Meyer or Kato-Ponce-Vega estimates \cite{Coifman-Meyer-1986,Coifman-McIntosh-Meyer-1982,Kato-Ponce-1988,Kenig-Ponce-Vega-1993} consider H\"older and Lipschitz-estimates.
In this section we show how this works by the extension method.
\begin{theorem}\label{th:KP}
Let $p \in (1,\infty)$ and $f,g, \varphi \in C_c^\infty(\R^n)$. Then,
\[
\| [\laph, \varphi](f)\|_{L^p(\R^n)} \aleq [\varphi]_{C^{0,1}} \|f\|_{L^p(\R^n)}
\]
or equivalently
\[
\int_{\R^n}  \brac{f\laph g - \laph  f\, g}\varphi  \aleq [\varphi]_{C^{0,1}} \|f\|_{L^p(\R^n)}\, \|g\|_{L^{p'}(\R^n)}.
\]
\end{theorem}
One observes readily that this estimate is completely trivial for $\laph$ replaced by the derivative $\nabla$.

\subsection{The integration by parts}
The main additional observation to start the integration by parts in this context is the following:

if $F: \R^{n+1}_+ \to \R$ is the harmonic extension of $f: \R^n \to \R$, then we  have the so-called Dirichlet-to-Neumann property  \[\partial_t F(x,0) = c \laph f\].

Thus, denoting $F,G,\Phi: \R^{n+1}_+ \to \R$ the harmonic extensions of $f,g,\varphi: \R^n \to \R$, then
\[
\begin{split}
&\left |\int_{\R^n}  \brac{f\laph g - \laph  f\, g}\, \varphi  \right |\\
=&\left |\int_{\R^{n+1}_+}  \partial_{t} \brac{\brac{F\partial_t G - \partial_tF\, G}\, \Phi } \right |.
\end{split}
\]
By a first cancellation we find readily
\[
\begin{split}
&\left |\int_{\R^n}  \brac{f\laph g - \laph  f\, g}\, \varphi  \right |\\
\leq &\left |\int_{\R^{n+1}_+}  \brac{F\partial_t G - \partial_tF\, G}\, \partial_t\Phi  \right |\\
&+\left |\int_{\R^{n+1}_+}  \brac{F\partial_{tt} G - \partial_{tt}F\, G}\, \Phi  \right |.
\end{split}
\]
By another integration by parts in $t$-direction, we find
\[
 \left |\int_{\R^n}  \brac{f\laph g - \laph  f\, g}\, \varphi  \right | \leq \mathcal{C}_1 + \mathcal{C}_2 + \mathcal{C}_3,
\]
where 
\[
\mathcal{C}_1 := \left |\int_{\R^{n+1}_+}  t\,  \brac{\brac{F\partial_t G - \partial_tF\, G}\, \partial_{tt}\Phi } \right |
\]
\[
\mathcal{C}_2 := \left |\int_{\R^{n+1}_+}  t\,  \partial_t \brac{F\partial_t G - \partial_tF\, G}\, \partial_{t}\Phi  \right |
\]
\[
\mathcal{C}_3 := \left |\int_{\R^{n+1}_+}  t\, \partial_t\brac{\brac{F\partial_{tt} G - \partial_{tt}F\, G}\, \Phi}  \right |
\]
We claim that we can estimate 
\begin{equation}\label{eq:KPgoal}
\begin{split}
&\mathcal{C}_1 + \mathcal{C}_2 + \mathcal{C}_3\\
\aleq& \int_{\R^{n+1}_+} t\, \brac{|F|\, |\nabla_{x,t} G|+|\nabla_{x,t} F|\, | G|}\,|\nabla_{x,t} \nabla_x \Phi| 
+\int_{\R^{n+1}_+} t\, |\nabla_{x,t} F|\, | \nabla_{x,t}G|\, | \nabla_x \Phi|  
\end{split}
\end{equation}
For $\mathcal{C}_1$ this is clear by the harmonicity of $\Phi$, $\partial_{tt} \Phi = -\lap_x \Phi$.

For $\mathcal{C}_2$ we find
\[
 \mathcal{C}_2 = \left |\int_{\R^{n+1}_+}  t\,  \brac{F\partial_{tt} G - \partial_{tt}F\, G}\, \partial_{t}\Phi  \right |
\]
Using the harmonicity of $F,G$ and the factoring the divergence,
\begin{equation}\label{eq:FpttG}
\begin{split}
 &F\partial_{tt} G - \partial_{tt}F\, G\\
 =&-\brac{F\lap_x G - \lap_x F\, G}\\
 =&-\nabla_x\cdot \brac{F\nabla_x G - \nabla_x F\, G}.
\end{split}
 \end{equation}
That is, an integration by parts in $x$-direction leads to
\[
 \mathcal{C}_2 = \left |\int_{\R^{n+1}_+}  t\,  \brac{F\nabla_x G - \nabla_x F\, G}\, \partial_{t}\nabla_{x}\Phi  \right |.
\]
This establishes the estimate \eqref{eq:KPgoal} for $\mathcal{C}_2$.

Using \eqref{eq:FpttG} in $\mathcal{C}_3$, 
\[
\mathcal{C}_3 = \left |\int_{\R^{n+1}_+}  t\, \partial_t\brac{\brac{F\nabla_x G - \nabla_x F\, G}\cdot\, \nabla_x \Phi}  \right |
\]
After computing the product rule for $\partial_t$ there is only one term not obviously satisfying the estimate \eqref{eq:KPgoal}, namely
\[
\begin{split}
 &\left |\int_{\R^{n+1}_+}  t\, \brac{F\nabla_x \partial_t G - \nabla_x \partial_t F\, G}\cdot\, \nabla_x \Phi  \right |\\
 \leq&\left |\int_{\R^{n+1}_+}  t\, \nabla_x \brac{F\partial_t G - \partial_t F\, G}\cdot\, \nabla_x \Phi  \right |\\
 &+\left |\int_{\R^{n+1}_+}  t\, \brac{\nabla_x F\partial_t G - \partial_t F\, \nabla_x G}\cdot\, \nabla_x \Phi  \right |\\
\end{split}
 \]
Thus, \eqref{eq:KPgoal} is established as well for $\mathcal{C}_3$.
\subsection{The trace estimates}
In \eqref{eq:KPgoal} it was established that for $f,g, \varphi \in C_c^\infty(\R^n)$ we have the following estimate for the respective harmonic extensions $F, G, \Phi:\R^{n+1}_+ \to \R$ 
\[
\begin{split}
&\int_{\R^n}  \brac{f\laph g - \laph  f\, g}\varphi  \\
\aleq& \int_{\R^{n+1}_+} t\, \brac{|F|\, |\nabla_{x,t} G|+|\nabla_{x,t} F|\, | G|}\,|\nabla_{x,t} \nabla_x \Phi| 
+\int_{\R^{n+1}_+} t\, |\nabla_{x,t} F|\, | \nabla_{x,t}G|\, | \nabla_x \Phi|  
\end{split}
\]
Theorem~\ref{th:KP} is then a consequence of the next two lemmas:
\begin{lemma}
Let $f,g, \varphi \in C_c^\infty(\R^n)$ then for the respective harmonic extensions $F, G, \Phi:\R^{n+1}_+ \to \R$, 
\[
 \int_{\R^{n+1}_+} t\, |\nabla_{x,t} F|\, | \nabla_{x,t}G|\, | \nabla_x \Phi| \aleq \|f\|_{L^{p}}\, \|g\|_{L^{p'}}\, [\varphi]_{\lip}.
\]
\end{lemma}
\begin{proof}
By, e.g., the maximum principle (one can also use an estimate by the maximal function similar to \eqref{eq:maxfctest})
\[
 \|\nabla_x \Phi\|_{L^\infty(\R^{n+1}_+} \leq \|\nabla_x \varphi\|_{L^\infty(\R^n)}.
\]
Note that there is no reason for this estimate to be true when $\|\nabla_x \Phi\|_{L^\infty(\R^{n+1}_+)}$ is replaced by $\|\partial_t \Phi\|_{L^\infty(\R^{n+1}_+)}$.

Thus, H\"older's inequality implies 
\[
 \int_{\R^{n+1}_+} t\, |\nabla_{x,t} F|\, | \nabla_{x,t}G|\, | \nabla_x \Phi| \aleq [\varphi]_{\lip} 
 \brac{\int_{\R^n} \brac{\int_{0}^\infty t |\nabla_{x,t} F|^2\, dt}^{\frac{p}{2}}\, dx}^{\frac{1}{p}}\, \brac{\int_{\R^n} \brac{\int_{0}^\infty t |\nabla_{x,t} G|^2\, dt}^{\frac{p'}{2}}\, dx}^{\frac{1}{p'}}
\]
Now as in the Sections before we find a square function, namely we can write 
\[
 \nabla_{x,t} F = t^{-1} \kappa_t \ast f,
\]
for a kernel $\kappa$ satisfying the square function estimate, and conclude that 
\[
 \brac{\int_{\R^n} \brac{\int_{0}^\infty t |\nabla_{x,t} F|^2\, dt}^{\frac{p}{2}}\, dx}^{\frac{1}{p}}
 = \brac{\int_{\R^n} \brac{\int_{0}^\infty |\kappa_t \ast F(x)|^2\, \frac{dt}{t}}^{\frac{p}{2}}\, dx}^{\frac{1}{p}} \aleq \|f\|_{L^p(\R^n)}
\]
and in the same fashion we have 
\[
 \brac{\int_{\R^n} \brac{\int_{0}^\infty t |\nabla_{x,t} G|^2\, dt}^{\frac{p'}{2}}\, dx}^{\frac{1}{p'}}
 \aleq \|g\|_{L^{p'}(\R^n)}.
\]
This proves the claim.\end{proof}

\begin{lemma}
Let $f,g, \varphi \in C_c^\infty(\R^n)$ then for the respective harmonic extensions $F, G, \Phi:\R^{n+1}_+ \to \R$, 
\[
 \int_{\R^{n+1}_+} t\, |F|\, |\nabla_{x,t} G|\,|\nabla_{x,t} \nabla_x \Phi| \aleq \|f\|_{L^{p}}\, \|g\|_{L^{p'}}\, [\nabla \varphi]_{BMO}.
\]
\end{lemma}
\begin{proof}
This is similar to the Jacobian estimate, Section~\ref{s:traceestjacob}: More precisely, by \eqref{eq:jacob1},
\[
\begin{split}
 &\int_{\R^{n+1}_+} t\, |F|\, |\nabla_{x,t} G|\,|\nabla_{x,t} \nabla_x \Phi|\\
 \leq &[\nabla_x \varphi]_{BMO} \int_{\R^n} \brac{\int_{|x-y|<t} |F(y,t)|^2\, |\nabla_{x,t} G(y,t)|^2 \frac{dy dt}{t^{n-1}} }^{\frac{1}{2}}dx
 \end{split}
 \]
By an estimate similar to \eqref{eq:maxfctest} we have
\[
 \sup_{(y,t):\, |x-y| < t} |F(y,t)| \aleq \mathcal{M}f(x).
\]
By H\"older inequality and the maximal theorem we thus obtain
\[
\begin{split}
 &\int_{\R^{n+1}_+} t\, |F|\, |\nabla_{x,t} G|\,|\nabla_{x,t} \nabla_x \Phi|\\
 \leq &[\nabla_x \varphi]_{BMO} \|f\|_{L^p(\R^n)}\, \brac{\int_{\R^n} \brac{\int_{|x-y|<t} |\nabla_{x,t} G(y,t)|^2 \frac{dy dt}{t^{n-1}} }^{\frac{p'}{2}}dx}^{\frac{1}{p'}}
 \end{split}
 \]
Again, we write
\[
 \nabla_{x,t} G =: t^{-1} \kappa_t \ast g,
\]
and use the non-tangential square function estimate \eqref{eq:ntsquare} to obtain
\[
\begin{split}
 &\brac{\int_{\R^n} \brac{\int_{|x-y|<t} |\nabla_{x,t} G(y,t)|^2 \frac{dy dt}{t^{n-1}} }^{\frac{p'}{2}}dx}^{\frac{1}{p'}}\\
 =& \brac{\int_{\R^n} \brac{\int_{|x-y|<t} |\kappa_t \ast g(y)|^2 \frac{dy dt}{t^{n+1}} }^{\frac{p'}{2}}dx}^{\frac{1}{p'}}\\
 \aleq& \|g\|_{L^{p'}(\R^n)}.
\end{split}
 \]
This establishes the claim.
\end{proof}

\section{On strengths and limitations of the method by harmonic extension}
In some sense, the extension method described above is similar to the Littlewood-Paley decomposition (which can be used to prove all of the statements alluded to above). One main advantage is that the technical argument of paraproducts can be avoided (at least for the commutators mentioned). But of course the mathematical deepness of the results means that the technical difficulties cannot disappear, they can just be shifted. While in the argument by Littlewood-Paley theory the space characterizations and compensation effects have to be dealt with at the same time, the argument by harmonic extension described here separates these two features: the compensation effects are observed from elementary computations (product rules and cancellations), and the spaces are characterized by trace spaces (which follow from quite deep facts from harmonic analysis). However, these trace space characterizations are independent of the specific commutator -- only the compensation phenomena change from commutator to commutator. 
Another limitations of the method by harmonic extension is that it is not clear how to treat, e.g. commutators involving general Calderon-Zygmund operators certain operators. Rather -- at least for the limit space estimates -- the extension needs to be adapted to the operators involved (which is relatively easy for simpler objects such as Riesz transforms, Riesz Potentials, and fractional Laplacians).

\section*{Acknowledgment}
This text is a result of an extended talk given at the ``International Workshop on Critical Phenomena'' at the  National Chiao Tung University. The author would like to thank the organizer D. Spector, the National Chiao Tung University Shing-Tung Yau Center, and the National Center for Theoretical Sciences for their hospitality.
\bibliographystyle{abbrv}
\bibliography{bib}%

\end{document}